\documentclass[12pt]{amsart}
\usepackage[margin=1.in, centering]{geometry}

\usepackage{amssymb,latexsym,amsmath,extarrows, mathrsfs, amsthm}
\usepackage{graphicx}

\usepackage[colorlinks, linkcolor=blue]{hyperref}

\usepackage{color}
\usepackage{wasysym}

\usepackage{float}

\newtheorem{theorem}{Theorem}[section]
\newtheorem{lemma}[theorem]{Lemma}
\newtheorem{proposition}[theorem]{Proposition}

\theoremstyle{remark}

\allowdisplaybreaks

\title{$k$-free lattice points in random walks}

\author{Kui Liu}
\address{School of Mathematics and Statistics, Qingdao University, 308 Ningxia Road, Shinan District, Qingdao, Shandong, China}
\email{liukui@qdu.edu.cn}

\author{ShunQi Ma}
\address{School of Mathematics and Statistics, Qingdao University, 308 Ningxia Road, Shinan District, Qingdao, Shandong, China}
\email{msqyzjy@163.com}

\keywords{random walk, $k$-free lattice points, visible lattice points}
\subjclass[2010]{60G50, 11H06, 11N37}
\begin{document}
\maketitle

\begin{abstract}
Let $\mathbb{Z}^2$ be the two-dimensional integer lattice. For an integer $k\geq 1$, a non-zero lattice point is $k$-free if the greatest common divisor of its coordinates is a $k$-free number. We consider the proportions of $k$-free and twin $k$-free lattice points on a path of an $\alpha$-random walker in $\mathbb{Z}^2$. Using the second-moment method and tools from analytic number theory, we prove that these two proportions are $1/\zeta(2k)$ and $\prod_{p}(1-2p^{-2k})$, respectively, where $\zeta$ is the Riemann zeta function and the infinite product takes over all primes.
\end{abstract}
\maketitle
\section{Introduction}
For an integer $k\geq 1$, a positive integer is said to be $k$-free if it is not divisible by any $k$-th power of primes. In the two-dimensional integer lattice $\mathbb{Z}^2$, we say a non-zero lattice point $(m,n)$ is $k$-free if $\gcd(m,n)$ is $k$-free, where $\gcd$ is the greatest common divisor function. Particularly, a non-zero lattice point $(m,n)\in\mathbb{Z}^2$ is $1$-free if and only if $\gcd(m,n)=1$, which is equivalent to $(m,n)$ is a visible lattice point (from the origin). In 2013, Pleasants and Huck \cite{PH} showed that the density of $k$-free lattice points in $\mathbb{Z}^2$ is $1/\zeta(2k)$, where $\zeta$ is the Riemann zeta function. We also refer to \cite{HB} for dynamical properties of $k$-free lattice points and \cite{BH,BMP,ASJM,PH} for related topics. 

In 2015, Cilleruelo, Fern{\'a}ndez, and Fern{\'a}ndez \cite{CFF} first considered visible (i.e. $1$-free) lattice points in $\mathbb{Z}^2$ from the viewpoint of random walk. For $0\textless \alpha \textless 1$, an $\alpha$-random walk starting at the origin of $\mathbb{Z}^2$ is defined by
\begin{align}\label{eq:the definition of random walk}
P_{i+1}=P_i+
\begin{cases}
	(1,0)& \text{with\ probability $\alpha$},\\
	(0,1)& \text{with\ probability $1-\alpha$},
	\end{cases}
\end{align}
for $i=0,1,2,\cdots$, where $P_i=(x_i,y_i)$ is the coordinate of the $\alpha$-random walker at the $i$-th step and $P_0=(0,0)$. This means, a random walker moving up and right in the integer lattice from the origin, with probabilities $\alpha$ and $1-\alpha$, respectively. Cilleruelo, Fern{\'a}ndez, and Fern{\'a}ndez showed that the proportion of $r$ consecutively visible lattice points on a path of an $\alpha$-random walker is almost surely $b_r(\alpha)\prod_{p\geq r}(1-rp^{-2})$, where $b_r(\alpha)$ is a polynomial in $\alpha$ with rational coefficients that can be explicitly computed. Particularly, they obtained $b_r(\alpha)=1$ for $r=1,2$. It follows that the asymptotic proportions of visible ($r=1$) and twin visible ($r=2$) lattice points on a path of an $\alpha$-random walker are independent of the probability $\alpha$. But for $r\geq 3$, this phenomenon does not hold.

In this paper, we generalize the above results of \cite{CFF} for $r=1,2$ to the case of $k$-free lattice points in $\mathbb{Z}^2$. It is also reasonable to consider the cases of $r\geq 3$, but here we will not go further in this direction.

For an $\alpha$-random walk given by~\eqref{eq:the definition of random walk}, consider a sequence  of random variables $\{X_i\}_{i\in\mathbb{N}}$ with
$$ X_i=\left\{
\begin{aligned}
	&1,\ \  P_i\ {\rm{is}}\ k{\rm{-free}}, \\
	&0,\  \ {\rm{otherwise}}.
\end{aligned}
\right.
$$
For $n\geq 1$, define
$$
\overline{S}_{n}=\overline{S}(n,k,\alpha):=\frac{X_1+X_2+\dots+X_n}{n},
$$
then $\overline{S}_{n}$ indicates the proportion of $k$-free lattice points in the first $n$ steps of an $\alpha$-random walk.
\begin{theorem}\label{thm:S_n=}
Under the above notations, for any $\alpha\in (0,1)$ and integer $k\geq 1$, we have
$$
\lim_{n\rightarrow +\infty}\overline{S}_{n}=\frac{1}{\zeta(2k)}
$$
almost surely, where $\zeta$ is the Riemann zeta function.
\end{theorem}

We remark that the limit proportion in Theorem \ref{thm:S_n=} is independent on the probability $\alpha$. Particularly, if $k=1$, then Theorem \ref{thm:S_n=} gives Theorem A of \cite{CFF}. We also remark that the limit proportion in Theorem \ref{thm:S_n=} is the same as the density of $k$-free lattice points in $\mathbb{Z}^2$.

For a path of an $\alpha$-random walk given by~\eqref{eq:the definition of random walk}, if two consecutive lattice points $P_i,P_{i+1}$ are both $k$-free, then we say $P_i,P_{i+1}$ are twin $k$-free lattice points. For $n\geq 1$, define
$$
\overline{T}_{n}=\overline{T}(n,k,\alpha):=\frac{X_1X_2+X_2X_3+\dots+X_{n}X_{n+1}}{n},
$$
then $\overline{T}_{n}$ indicates the proportion of twin $k$-free lattice points in the first $n+1$ steps of an $\alpha$-random walk.
\begin{theorem}\label{thm:T_n=}
Under the above notations, for any $0<\alpha<1$ and integer $k\geq 1$, we have
$$
\lim_{n\rightarrow+\infty}\overline{T}_{n}=\prod_{p}\Big(1-\frac{2}{p^{2k}}\Big)
$$
almost surely, where the infinite product takes over all primes.
\end{theorem}

Again, the limit proportion in Theorem \ref{thm:T_n=} is independent on the probability $\alpha$.

\subsection{Notations}
As usual, for real functions $f$ and $g$, we use the expressions $f=O(g)$ and $f \ll g$ to mean $|f| \leq Cg$ for some constant $C>0$. When this constant $C$ depends on some parameter $\varepsilon$, we write $f\ll_{\varepsilon}g$ and $f=O_{\varepsilon}(g)$. We use $\mathbb{R}$, $\mathbb{Z}$ and $\mathbb{N}$ to denote the sets of all real numbers, integers and positive integers, respectively. Moreover, we use $\mathbb{P}$, $\mathbb{E}$ and $\mathbb{V}$ to denote taking probability, expectation and variance, respectively.

\section{Preliminaries}
\subsection{Tools from probability theory}
We need the following two results from probability theory. Lemma \ref{lemma:sum} is the second-moment method (see Lemma 2.5, \cite{CFF}) and Lemma \ref{lemma:The local  central limit theorem} is the local central limit theorem (see Theorem 3.5.2, \cite{D}). 
\begin{lemma}\label{lemma:sum}
For a sequence of uniformly bounded random variables $(W_i)_{i\geq1}$, let $\overline{S}_{n}=(W_1+\dots+W_n)/n$. If the expectation $\mathbb{E}(\overline{S}_n)$ and the variance $\mathbb{V}(\overline{S}_{n})$ of $\overline{S}_{n}$ satisfy
$$
\lim_{n\to \infty}\mathbb{E}(\overline{S}_n)=\mu
$$
and
$$
\mathbb{V}(\overline{S}_{n})\ll_{\delta} n^{-\delta}
$$
for some $\delta>0$ and any $n\geq 1$, then we have
$$
\lim_{n\to \infty}\overline{S}_{n}=\mu
$$
almost surely.
\end{lemma}

\begin{lemma}\label{lemma:The local central limit theorem}
Let $\alpha\in (0,1)$ be fixed. For any integer $n\geq 1$, we have 
\begin{align*}
\max_{0\leq l\leq n}{n \choose  l}\alpha^l(1-\alpha)^{n-l}=O_{\alpha}\Big( \frac{1}{\sqrt n}\Big), 
\end{align*}
where ${n \choose  l},\ 0\leq l\leq n$ are binomial coefficients. 
\end{lemma}
\subsection{Divisor functions}

For $l\geq 2$, let
$$
\tau_l(n):=\sum\limits_{n=d_1d_2...d_l}1
$$ be the $l$-fold divisor function. By (1.81) of \cite{HE}, we have 
\begin{align}\label{eq: bound for tau_l(n)}
\tau_l(n)=O_{l,\varepsilon}( n^\varepsilon)
\end{align}
for any $\varepsilon>0$. By the above upper bound and partial summation, we have the following estimates involving the $3$-fold divisor function, which will be used in the proof of our theorems.

\begin{lemma}\label{lem: estimetes involving tau_3}
For any integer $n\geq 1$ and any $\varepsilon>0$, we have
$$
\sum\limits_{1\leq i< j\leq n}\frac{\tau_3(j)\tau_3(i)}{\sqrt{i}}=O_{\varepsilon}(n^{3/2+\varepsilon})\quad {\rm and}\quad \sum\limits_{1\leq i< j\leq n}\frac{\tau_3(j)}{\sqrt{j-i}}=O_{\varepsilon}(n^{3/2+\varepsilon}).
$$
\end{lemma}

\begin{lemma}\label{lem: estimetes (2) involving tau_3}
For any integer $n\geq 1$ and any $\varepsilon>0$, we have
$$
\sum_{1\leq i \leq n}\sum_{i+1<j\leq n}\frac{\tau_3(i)\tau_3(i+1)\tau_3(j)\tau_3(j+1)}{\sqrt{i}}=O_{\varepsilon}(n^{3/2+\varepsilon})
$$
and
$$
\sum_{1\leq i \leq n}\sum_{i+1<j\leq n}\frac{\tau_3(j)\tau_3(j+1)}{\sqrt{j-i-1}}=O_{\varepsilon}(n^{3/2+\varepsilon}).
$$
\end{lemma}
The factor $n^{\varepsilon}$ in Lemmas \ref{lem: estimetes involving tau_3} and \ref{lem: estimetes (2) involving tau_3} can be replaced by a power of $\log n$, but this is not necessary here.

We also use the unitary divisor function $\tau^*(n)$, which is defined by
\begin{align}\label{def: unitary divisor function}
\tau^*(n):=\sum\limits_{\substack{n=n_1n_2\\ \gcd(n_1,n_2)=1}}1.
\end{align}
Note that $\tau^*(n)$ is multiplicative and
\begin{align}\label{eq: values of tau*(pm)}
\tau^*(p^m)=2
\end{align}
for any prime power $p^m$. Obviously, for $n\geq 1$ we have
\begin{align}\label{eq: bound for tau*(n)}
|\tau^*(n)|\leq \tau_2(n)=O_{\varepsilon}(n^{\varepsilon}).
\end{align}

\subsection{Estimates for sums of binomial probabilities}
To prove our theorems, we also need two results for sums of binomial probabilities. The following result is Lemma 2.1 of \cite{CFF}.
\begin{lemma}\label{lemma: the sum of mod}
For any $\alpha\in (0,1)$, there is a constant $B_\alpha\textgreater0$ such that for integers $n\geq 1$,  $1\leq d\leq n$ and $r\in\left\lbrace 0,1,\dots,d-1\right\rbrace $, there holds
$$\left| \sum_{l\equiv r\bmod d}
{n \choose  l}\alpha^l(1-\alpha)^{n-l}-\frac{1}{d}\right| \leq \frac{B_\alpha}{\sqrt{n}}.
$$
\end{lemma}

Using Lemma \ref{lemma: the sum of mod}, we derive the following result.
\begin{lemma}\label{lemma:mean1}
For any $\alpha\in(0,1)$ and integer $l\geq 1$, let $u_j, 1\leq j\leq l$ be pairwise coprime positive integers. If positive integers $d_j$ satisfy $d_j\mid u_j$ for any $1\leq j\leq l$, then for $n\geq 1$ and any integers $a_1,\cdots,a_j$, we have
$$
\sum_{\substack{0\leq s\leq n\\ \gcd(s+a_j,u_j)=d_j,\\1\leq j\leq l}}{n \choose s}\alpha^s(1-\alpha)^{n-s}=\frac{1}{d_1\cdots d_k}\sum\limits_{\substack{r_j\mid (u_j/d_j),\\1\leq j\leq l}}\frac{\mu(r_1)\cdots\mu(r_l)}{r_1\cdots r_l}+O\Big(\frac{1}{\sqrt{n}}\prod\limits_{1\leq j\leq l}\tau_2(u_j/d_j)\Big),
$$
where the implied $O$-constant depends only on $\alpha$ and $l$.
\end{lemma}
\begin{proof}
For simplicity, denote the left hand side of the above equation by
$$
S:=\sum_{\substack{0\leq s\leq n\\ \gcd(s+a_j,u_j)=d_j,1\leq j\leq l}}C_{\alpha}(n,s), 
$$
where
\begin{align}\label{eq: C_alpha(n,s)=}
C_{\alpha}(n,s):={n \choose s}\alpha^s(1-\alpha)^{n-s}.
\end{align}
Then we have
$$
S=\sum_{\substack{0\leq s\leq n\\ s\equiv -a_j\bmod d_j,\\ \gcd((s+a_j)/d_j,u_j/d_j)=1,\\ 1\leq j\leq l}}C_{\alpha}(n,s).
$$
Applying the formula
\begin{align*}
\sum\limits_{r\mid n}\mu(r)=
\begin{cases}
1,\quad &n=1,\\
0,\quad &{\rm otherwise},
\end{cases}
\end{align*}
where $\mu$ is the M\"{o}bius function, we obtain
$$
S=\sum_{\substack{0\leq s\leq n\\ s\equiv -a_j\bmod d_j,\ 1\leq j\leq l}}C_{\alpha}(n,s)\prod_{1\leq j\leq l}\bigg(\sum\limits_{r_j\mid \gcd((s+a_j)/d_j,u_j/d_j)}\mu(r_j)\bigg).
$$
Changing the order of summations, we obtain
$$
S=\sum\limits_{r_j\mid (u_j/d_j),1\leq j\leq l}\mu(r_1)\cdots\mu(r_l)\sum_{\substack{0\leq s\leq n\\ s\equiv -a_j\bmod r_jd_j,\ 1\leq j\leq l}}C_{\alpha}(n,s).
$$
Note that $u_1,\cdots,u_l$ are pairwise coprime and $r_jd_j\mid u_j$ for $1\leq j\leq l$. Then we have $r_1d_1,\cdots,r_ld_l$ are pairwise coprime. By the Chinese remainder theorem, there exists an integer
$$
a\in\Big\{0,\cdots,\prod\limits_{1\leq j\leq l}r_jd_j-1\Big\}
$$
such that 
$$
S=\sum\limits_{r_j\mid (u_j/d_j),1\leq j\leq l}\mu(r_1)\cdots\mu(r_l)\sum_{\substack{0\leq s\leq n\\ s\equiv a\bmod \prod\limits_{1\leq j\leq l}r_jd_j}}C_{\alpha}(n,s).
$$
Breaking the sum $S$ into two sums according to $\prod_{1\leq j\leq l}{r_jd_j}\leq n$ or not, we write
\begin{align}\label{eq: S=S1+S2}
S=S_1+S_2,
\end{align}
where
$$
S_1=\sum\limits_{\substack{r_j\mid (u_j/d_j),1\leq j\leq l\\ \prod\limits_{1\leq j\leq l}{r_jd_j}\leq  n}}\mu(r_1)\cdots\mu(r_l)\sum_{\substack{0\leq s\leq n\\ s\equiv a\bmod \prod\limits_{1\leq j\leq l}r_jd_j}}C_{\alpha}(n,s)
$$
and
$$
S_2=\sum\limits_{\substack{r_j\mid (u_j/d_j),1\leq j\leq l\\ \prod\limits_{1\leq j\leq l}{r_jd_j}>n}}\mu(r_1)\cdots\mu(r_l)\sum_{\substack{0\leq s\leq n\\ s\equiv a\bmod \prod\limits_{1\leq j\leq l}r_jd_j}}C_{\alpha}(n,s).
$$
For $S_1$, by Lemma \ref{lemma: the sum of mod}, we obtain
$$
S_1=\frac{1}{d_1\cdots d_l}\sum\limits_{\substack{r_j\mid (u_j/d_j),1\leq j\leq l\\ \prod\limits_{1\leq j\leq l}{r_jd_j}\leq  n}}\frac{\mu(r_1)\cdots\mu(r_l)}{r_1\cdots r_l}+O_{\alpha,l}\Big(\frac{1}{\sqrt{n}}\prod\limits_{1\leq j\leq l}\tau_2(u_j/d_j)\Big).
$$
The first term in the above equation is equal to
$$
\frac{1}{d_1\cdots d_l}\sum\limits_{r_j\mid (u_j/d_j),1\leq j\leq l}\frac{\mu(r_1)\cdots\mu(r_l)}{r_1\cdots r_l}+O_{\alpha,l}\Big(\frac{1}{n}\prod\limits_{1\leq j\leq l}\tau_2(u_j/d_j)\Big),
$$
since
$$
\sum\limits_{\substack{r_j\mid (u_j/d_j),1\leq j\leq l\\ \prod\limits_{1\leq j\leq l}{r_jd_j}>n}}\frac{\mu(r_1)\cdots\mu(r_l)}{r_1\cdots r_l}\ll_l \frac{d_1\cdots d_l}{n}\prod\limits_{1\leq j\leq l}\tau_2(u_j/d_j). 
$$
It follows that
\begin{align}\label{eq: S_1=result}
S_1=\frac{1}{d_1\cdots d_l}\sum\limits_{r_j\mid (u_j/d_j),1\leq j\leq l}\frac{\mu(r_1)\cdots\mu(r_l)}{r_1\cdots r_l}+O_{\alpha,l}\Big(\frac{1}{\sqrt{n}}\prod\limits_{1\leq j\leq l}\tau_2(u_j/d_j)\Big).
\end{align}
For $S_2$, by Lemma \ref{lemma:The local central limit theorem}, we obtain
\begin{align}\label{eq: S_2<<result}
S_2=O_{\alpha,l}\Big(\frac{1}{\sqrt{n}}\prod\limits_{1\leq j\leq l}\tau_2(u_j/d_j)\Big). 
\end{align}
Now our desired result follows from inserting \eqref{eq: S_1=result} and \eqref{eq: S_2<<result} into \eqref{eq: S=S1+S2}.
\end{proof}
\subsection{Two arithmetic functions}
For $k\geq 1$, define
\begin{align}\label{eq:definition of g_k(n)}
g_k(n):=\sum_{\substack{rd=n\\d\ {\rm{is}} \ k-{\rm{free}}}}\mu (r),
\end{align}
where $\mu$ is the M\"{o}bius function. Note that $g_k(n)$ is multiplicative and
$$
g_k(p^m)=\begin{cases}
-1,\quad m=k,\\
0,\quad {\rm otherwise}, 
\end{cases}
$$
for any prime power $p^m$. It follows that
\begin{align}\label{eq: euler product of gk/n^2}
\sum_{n=1}^{\infty}\frac{g_k(n)}{n^2}=\prod_{p}\Big(1-\frac{1}{p^{2k}}\Big)=\frac{1}{\zeta(2k)}.
\end{align}
Here and in the following, the symbol $\prod_p$ always means taking product over all primes. Similarly, since $\tau^*(n)$ is also multiplicative, we have
\begin{align}\label{eq: euler product of gk tau*/n^2}
\sum_{n=1}^{\infty}\frac{g_k(n)\tau^*(n)}{n^2}=\prod_{p}\Big(1-\frac{2}{p^{2k}}\Big),
\end{align}
where we have used \eqref{eq: values of tau*(pm)}. Obviously, for $n\geq 1$ we have
\begin{align}\label{eq: bound for gk(n)}
|g_k(n)|\leq \tau_2(n)=O_{\varepsilon}(n^{\varepsilon}).
\end{align}

For $k\geq 1$, define
\begin{align}\label{eq: definition of f_k(n)}
f_k(n)=\sum_{\substack{rd\mid n\\d\ {\rm{is}}\ k-{\rm{free}}}}\frac{\mu (r)}{rd}.
\end{align}
Obviously, for $n\geq 1$ we have
\begin{align}\label{eq: bound for fk(n)}
|f_k(n)|\leq \tau_3(n)=O_{\varepsilon}(n^{\varepsilon}).
\end{align}

\begin{lemma}\label{lem: mean of f_k(n)}
For any $\varepsilon>0$, we have
$$
\sum\limits_{1\leq n\leq N}f_k(n)=\frac{N}{\zeta(2k)}+O_{\varepsilon}\big(N^{\varepsilon}\big)
$$
for $N\geq 1$.
\end{lemma}
\begin{proof}
Let $rd=w$ in \eqref{eq: definition of f_k(n)}. Then we have
\begin{align}\label{eq:f_k(i) and g_k(w)}
f_k(n)=\sum_{w\mid n}\frac{1}{w}\sum_{\substack{rd=w\\d\ {\rm{is}} \ k-{\rm{free}}}}\mu (r)=\sum_{w\mid n}\frac{g_k(w)}{w}.
\end{align}
It follows that
$$
\sum\limits_{1\leq n\leq N}f_k(n)=\sum\limits_{1\leq n\leq N}\sum_{w\mid n}\frac{g_k(w)}{w}=\sum_{w\leq N}\frac{g_k(w)}{w}\sum_{\substack{1\leq n\leq N\\ n\equiv 0 \bmod w}}1,
$$
where we have changed the order of summations. It follows that
\begin{align*}
\sum\limits_{1\leq n\leq N}f_k(n)=\sum_{w\leq N}\frac{g_k(w)}{w}\Big(\frac{N}{w}+O(1) \Big), 
\end{align*}
which gives
$$
\sum\limits_{1\leq n\leq N}f_k(n)=N\sum_{w\leq N}\frac{g_k(w)}{w^2}+O\Big(\sum_{w\leq N}\frac{|g_k(w)|}{w}\Big).
$$
Using bound \eqref{eq: bound for gk(n)} to estimate the $O$-term, we obtain
$$
\sum\limits_{1\leq n\leq N}f_k(n)=N\sum_{w\leq N}\frac{g_k(w)}{w^2}+O_{\varepsilon}(N^{\varepsilon}).
$$
Extending the range of the sum over $w$, we obtain
$$
\sum\limits_{1\leq n\leq N}f_k(n)=N\sum_{w=1}^{\infty}\frac{g_k(w)}{w^2}+O\Big(N\sum_{w>N}\frac{|g_k(w)|}{w^2}\Big)+O_{\varepsilon}(N^{\varepsilon}).
$$
Using bound \eqref{eq: bound for gk(n)} again to estimate the first $O$-term in the above, we obtain
$$
\sum\limits_{1\leq n\leq N}f_k(n)=N\sum_{w=1}^{\infty}\frac{g_k(w)}{w^2}+O_{\varepsilon}(N^{\varepsilon}).
$$
This together with \eqref{eq: euler product of gk/n^2} gives our desired result.
\end{proof}

\begin{lemma}\label{lem: mean of f_k(n)f_k(n+1)}
For any $\varepsilon>0$, we have
$$
\sum\limits_{1\leq n\leq N}f_k(n)f_k(n+1)=N\prod_{p}\Big(1-\frac{2}{p^{2k}}\Big)+O_{\varepsilon}(N^{\varepsilon})
$$
for $N\geq 1$.
\end{lemma}

\begin{proof}

It follows from \eqref{eq:f_k(i) and g_k(w)} that  
\begin{align*}
\sum\limits_{1\leq n\leq N}f_k(n)f_k(n+1)=\sum\limits_{1\leq n\leq N}\sum_{w_1\mid n}\frac{g_k(w_1)}{w_1}\sum_{w_2\mid( n+1)}\frac{g_k(w_2)}{w_2}.
\end{align*}
Changing the order of summations, we obtain
$$
\sum\limits_{1\leq n\leq N}f_k(n)f_k(n+1)=\sum_{\substack{w_1\leq N,w_2\leq N+1}}\frac{g_k(w_1)g_k(w_2)}{w_1w_2}\sum_{\substack{1\leq n\leq N\\ n\equiv 0\bmod w_1\\ n\equiv -1\bmod w_2}}1.
$$
It follows from the Chinese remainder theorem that
\begin{align*}
\sum\limits_{1\leq n\leq N}f_k(n)f_k(n+1)=\sum_{\substack{w_1\leq N,w_2\leq N+1\\(w_1,w_2)=1}}\frac{g_k(w_1)g_k(w_2)}{w_1w_2}\Big(\frac{N}{w_1w_2}+O(1)\Big),
\end{align*}
which gives
$$
\sum\limits_{1\leq n\leq N}f_k(n)f_k(n+1)=N\sum_{\substack{w_1\leq N,w_2\leq N+1\\(w_1,w_2)=1}}\frac{g_k(w_1)g_k(w_2)}{(w_1w_2)^2}	+O\Big(\sum_{w_1\leq N}\frac{|g_k(w_1)|}{w_1}\sum_{w_2\leq N+1}\frac{|g_k(w_2)|}{w_2}\Big).  
$$
Using bound \eqref{eq: bound for gk(n)} to estimate the $O$-term, we obtain
\begin{align*}
\sum\limits_{1\leq n\leq N}f_k(n)f_k(n+1)=N\sum_{\substack{w_1\leq N,w_2\leq N+1\\(w_1,w_2)=1}}\frac{g_k(w_1)g_k(w_2)}{(w_1w_2)^2}+O_{\varepsilon}(N^{\varepsilon}).
\end{align*}
Since $g_k(n)$ is multiplicative, we have
$$
\sum\limits_{1\leq i\leq n}f_k(n)f_k(n+1)=N\sum_{\substack{w_1\leq N,w_2\leq N+1\\(w_1,w_2)=1}}\frac{g_k(w_1w_2)}{(w_1w_2)^2}+O_{\varepsilon}(N^{\varepsilon})
$$
Letting $w=w_1w_2$, we have
$$
\sum\limits_{1\leq i\leq n}f_k(n)f_k(n+1)=N\sum_{w\leq N(N+1)}\frac{g_k(w)\tau^*(w;N)}{w^2}+O_{\varepsilon}(N^{\varepsilon}), 
$$
where
$$
\tau^*(w;N)=\sum\limits_{\substack{w=w_1w_2\\ (w_1,w_2)=1\\w_1\leq N,w_2\leq N+1}}1.
$$
Note that for $w\leq N$, we have $\tau^*(w;N)=\tau^*(n)$, where $\tau^*(n)$ is given by \eqref{def: unitary divisor function}. Then we have
\begin{align*}
\sum\limits_{1\leq i\leq n}f_k(n)f_k(n+1)&=N\sum_{w\leq N}\frac{g_k(w)\tau^*(w)}{w^2}+N\sum_{w>N}\frac{g_k(w)\tau^*(w;N)}{w^2}+O_{\varepsilon}(N^{\varepsilon})\\
&=N\sum_{w\leq N}\frac{g_k(w)\tau^*(w)}{w^2}+O_{\varepsilon}(N^{\varepsilon}),
\end{align*}
where we have used bounds $\tau^*(w;N)\leq \tau^*(w)$, \eqref{eq: bound for tau*(n)}, \eqref{eq: bound for gk(n)} and
$$
\sum_{w>N}\frac{g_k(w)\tau^*(w;N)}{w^2}\ll \sum_{w>N}\frac{\tau^2(w)}{w^2}\ll N^{-1+\varepsilon}
$$
for any $\varepsilon>0$. Extending the range of the sum over $w$, we obtain
$$
\sum\limits_{1\leq n\leq N}f_k(n)f_k(n+1)=N\sum_{w=1}^{\infty}\frac{g_k(w)\tau^*(w)}{w^2}+O\bigg(N\sum_{w>N}\frac{g_k(w)\tau^*(w)}{w^2}\bigg)+O_{\varepsilon}(N^{\varepsilon}). 
$$
With the help of bounds \eqref{eq: bound for tau*(n)} and \eqref{eq: bound for gk(n)} again, we estimate the first $O$-term in the above and obtain
$$
\sum\limits_{1\leq n\leq N}f_k(n)f_k(n+1)=N\sum_{w=1}^{\infty}\frac{g_k(w)\tau^*(w)}{w^2}+O_{\varepsilon}(N^{\varepsilon}). 
$$
This together with \eqref{eq: euler product of gk tau*/n^2} gives our desired result.
\end{proof}


\section{Proof of Theorem \ref{thm:S_n=}}\label{sec:3}
According to Lemma \ref{lemma:sum}, to prove Theorem \ref{thm:S_n=}, it is sufficient to compute the expectation and estimate the variance of $\overline{S}_n$, respectively. We first prove the following result.
\begin{lemma}\label{lem: mean of C_alpha(m,n) with k-free condition}
Let $0<\alpha<1$ is fixed. For any integers $k\geq 1$ and $a,b,n$ with $b,n\in\mathbb{N}$, we have
$$
\sum_{\substack{0\leq m\leq n\\ \gcd(m+a,b)\ {\rm{is}}\ k-{\rm{free}}}}C_{\alpha}(n,m)=f_k(b)+O_\alpha\Big(\frac{\tau_3(b)}{\sqrt{n}}\Big), 
$$
where $C_{\alpha}(n,m)$ is given by \eqref{eq: C_alpha(n,s)=}.
\end{lemma}

\begin{proof}
For simplicity, denote the left hand side of the above equation by $R$. Then we have
$$
R:=\sum_{\substack{d\mid b\\d\ {\rm{is}}\ k-{\rm{free}}}}\sum_{\substack{0\leq m\leq n\\ \gcd(m+a,b)=d}}C_{\alpha}(n,m).
$$
Applying Lemma \ref{lemma:mean1} to the sum over $m$, we obtain
$$
R=\sum_{\substack{d\mid b\\d\ {\rm{is}}\ k-{\rm{free}}}}\bigg(\sum_{rd\mid b}\frac{\mu (r)}{rd}+O_\alpha\Big(\frac{\tau_2(b/d)}{\sqrt{n}}\Big)\bigg)
$$
The contribution of the $O$-term to $R$ is
$$
\ll_{\alpha} \frac{1}{\sqrt{n}}\sum_{d\mid b}\tau_2(b/d)=\frac{\tau_3(b)}{\sqrt{n}}.
$$
Hence, we have
$$
R=f_k(b)+O_\alpha\Big(\frac{\tau_3(b)}{\sqrt{n}}\Big),
$$
where $f_k(b)$ is given by \eqref{eq: definition of f_k(n)}. This completes our proof.
\end{proof}

\subsection{The expectation of $\overline{S}_n$}
In this subsection, we compute the expectation $\mathbb{E}(\overline{S}_n)$ and prove the following result.
\begin{proposition}\label{prop: mean of Sn}
For integer $k\geq 1$ and $\alpha\in (0,1)$, we have
$$
\mathbb{E}(\overline{S}_n)=\frac{1}{\zeta(2k)}+O_{k,\alpha}(n^{-1/2+\varepsilon})
$$
for any $\varepsilon>0$ and $n\geq 1$.
\end{proposition}

To prove Proposition \ref{prop: mean of Sn}, we first write
\begin{align}\label{eq: E(Sn)=(1)}
\mathbb{E}(\overline{S}_n)=\frac{1}{n}\sum_{1\leq i\leq n}\mathbb{E}(X_i).
\end{align}
For each $1\leq i\leq n$, by the definition of $X_i$, we have
\begin{align*}
\mathbb{E}(X_i)=\mathbb{P}(P_i\ {\rm{is}}\  k-{\rm{free}}),
\end{align*}
where $P_i=(x_i,y_i)$ is the coordinate of the $\alpha$-random walker at the $i$-th step. Observe that $x_i+y_i=i$, thus we can write $P_i=(l,i-l)$ for some $l=0,1,\dots ,i$. The probability that $P_i=(l,i-l)$ is
$$
\mathbb{P}\big(P_i=(l,i-l)\big)=C_{\alpha}(i,l),
$$
where $C_{\alpha}(i,l)$ is given by \eqref{eq: C_alpha(n,s)=}. 
By Lemma \ref{lem: mean of C_alpha(m,n) with k-free condition} and $\gcd(l,i-l)=\gcd(l,i)$, we have
\begin{align}\label{eq:E(Xi)}
\mathbb{E}(X_i)=\sum_{\substack{0\leq l\leq i\\ \gcd(l,i)\ {\rm{is}}\ k-{\rm{free}}}}C_{\alpha}(i,l)=f_k(i)+O_\alpha\Big(\frac{\tau_3(i)}{\sqrt{i}}\Big),
\end{align}
where $f_k$ is given by \eqref{eq: definition of f_k(n)}. Inserting \eqref{eq:E(Xi)} to (\ref{eq: E(Sn)=(1)}), we obtain
$$
\mathbb{E}(\overline{S}_n)=\frac{1}{n}\sum_{1\leq i\leq n}f_k(i)+O_\alpha\Big(\frac{1}{n}\sum\limits_{1\leq i\leq n}\frac{\tau_3(i)}{\sqrt{i}}\Big). 
$$
Using bound \eqref{eq: bound for tau_l(n)} to estimate the $O$-term, we obtain
\begin{align}\label{eq: mean of S_n=(1)}
\mathbb{E}(\overline{S}_n)=\frac{1}{n}\sum_{1\leq i\leq n}f_k(i)+O_\alpha(n^{-1/2+\varepsilon})
\end{align}
for any $\varepsilon>0$. This together with Lemma \ref{lem: mean of f_k(n)} yields Proposition \ref{prop: mean of Sn}.

\subsection{Estimating the variance of $\overline{S}_n$}
In this subsection, we estimate the variance of $\overline{S}_n$ and prove the following result.
\begin{proposition}\label{prop: variance of Sn}
For integer $k\geq 1$ and $\alpha\in (0,1)$, we have
$$
\mathbb{V}(\overline{S}_n)=O_{k,\alpha}(n^{-1/2+\varepsilon})
$$
for any $\varepsilon>0$ and $n\geq 1$.
\end{proposition}

To prove Proposition \ref{prop: variance of Sn}, we first write
$$
\mathbb{V}(\overline{S}_n)=\mathbb{E}(\overline{S}_n^2)-\mathbb{E}^2(\overline{S}_n).
$$	
Further, by the definition of $\overline{S}_n$, we have
\begin{equation}\label{eq:V(S_n)}
\mathbb{V}(\overline{S}_n)=\frac{1}{n^2}\sum_{1\leq i\leq n}\mathbb{E}({X_i}^2)+\frac{2}{n^2}\sum_{1\leq i\textless j\leq n}\mathbb{E}(X_iX_j)-\frac{1}{n^2}\mathbb{E}^2\Big(\sum_{1\leq i\leq n}{X_i}\Big).
\end{equation}

For the first term on the right hand side of \eqref{eq:V(S_n)}, since $\mathbb{E}(X_i^2)=\mathbb{E}(X_i)$  for each $1\leq i\leq n$ by the definition of $X_i$. Then we have 
\begin{align}\label{eq:E(X_i^2)}
\sum_{1\leq i\leq n}\mathbb{E}({X_i}^2)=\sum_{1\leq i\leq n}\mathbb{E}({X_i})=O(n),
\end{align}
where we have used $\mathbb{E}(X_i)=\mathbb{P}(X_i)=O(1)$ for any $1\leq i\leq n$.

For the third term on the right hand side of \eqref{eq:V(S_n)}, by the definition of $\overline{S}_n$ and \eqref{eq: mean of S_n=(1)}, we obtain
\begin{align}\label{eq: sum E^2sum Xi=}
\mathbb{E}^2\Big(\sum_{1\leq i\leq n}{X_i}\Big)=\Big(\sum_{1\leq i\leq n}f_k(i)\Big)^2+O_{\varepsilon}(n^{3/2+\varepsilon}),
\end{align}
for any $\varepsilon>0$, where we have used Lemma \ref{lem: mean of f_k(n)} to obtain the $O$-term in the above.

Now we deal with the second term on the right hand side of \eqref{eq:V(S_n)}. For $1\leq i< j\leq n$, let $P_i$ and $P_j$ be the coordinates of the $i$-th and $j$-th steps of a path of the $\alpha$-random walker, respectively. Here, we remark that $P_j$ depends on $P_i$. By the definition of $X_i$, we have
\begin{align*}
\mathbb{E}(X_iX_j)=\mathbb{P}(P_i,P_j\ {\rm{are\ both}}\  k-{\rm{free}}).
\end{align*}
Note that $P_i=(l,i-l)$ for some $0\leq l\leq i$, then we have $P_j=(l+m,j-l-m)$ for some $0\leq m\leq j-i$. 
The probability that $P_i$ and $P_j$ are both $k$-free is
$$
\sum\limits_{\substack{0\leq l\leq i\\\gcd(l,i-l)\ {\rm is}\ k-{\rm free}}}\ \sum\limits_{\substack{0\leq m\leq j-i\\\gcd(l+m,j-l-m)\ {\rm is}\ k-{\rm free}}}\mathbb{P}\big(P_i=(l,i-l),P_j=(l+m,j-l-m)\big).
$$
Note that
$$
\mathbb{P}\big(P_i=(l,i-l),P_j=(l+m,j-l-m)\big)=C_{\alpha}(i,l)C_{\alpha}(j-i,m).
$$
Since $\gcd(l,i-1)=\gcd(l,i)$ and $\gcd(l+m,j-l-m)=\gcd(l+m,j)$, then we have
\begin{align}\label{eq: E(XiXj)=}
\mathbb{E}(X_iX_j)=\sum\limits_{\substack{0\leq l\leq i\\\gcd(l,i)\ {\rm is}\ k-{\rm free}}}C_{\alpha}(i,l)T_k(l,i,j,\alpha),
\end{align}
where
$$
T_k(l,i,j,\alpha)=\sum_{\substack{0\leq m\leq j-i\\\gcd(l+m,j)\ {\rm{is}}\ k-{\rm{free}}}}C_{\alpha}(j-i,m).
$$
For $T_k(l,i,j,\alpha)$, applying Lemma \ref{lem: mean of C_alpha(m,n) with k-free condition} to the sum over $m$, we obtain
\begin{align}\label{eq: T_k=}
T_k(l,i,j,\alpha)=&f_k(j)+O_{\alpha}\Big(\frac{\tau_3(j)}{\sqrt{j-i}}\Big),
\end{align}
where $f_k(j)$ is given by \eqref{eq: definition of f_k(n)}. Inserting \eqref{eq: T_k=} to \eqref{eq: E(XiXj)=}, we obtain
$$
\mathbb{E}(X_iX_j)=\sum\limits_{\substack{0\leq l\leq i\\\gcd(l,i)\ {\rm is}\ k-{\rm free}}}C_{\alpha}(i,l)\bigg(f_k(j)+O_{\alpha}\Big(\frac{\tau_3(j)}{\sqrt{j-i}}\Big)\bigg).
$$
By the binomial theorem, the contribution of the $O$-term to $\mathbb{E}(X_iX_j)$ is
$O(\tau_3(j)/\sqrt{j-i})$. Hence, we have
$$
\mathbb{E}(X_iX_j)=f_k(j)\sum\limits_{\substack{0\leq l\leq i\\\gcd(l,i)\ {\rm is}\ k-{\rm free}}}C_{\alpha}(i,l)+O_{\alpha}\Big(\frac{\tau_3(j)}{\sqrt{j-i}}\Big).
$$
Using Lemma \ref{lem: mean of C_alpha(m,n) with k-free condition} again, we obtain
\begin{align*}
\mathbb{E}(X_iX_j)=&f_k(j)\bigg(f_k(i)+O_{\alpha}\Big(\frac{\tau_3(i)}{\sqrt{i}}\Big)\bigg)+O_{\alpha}\Big(\frac{\tau_3(j)}{\sqrt{j-i}}\Big)\\
=&f_k(i)f_k(j)+O_{\alpha}\Big(\frac{\tau_3(j)\tau_3(i)}{\sqrt{i}}\Big)+O_{\alpha}\Big(\frac{\tau_3(j)}{\sqrt{j-i}}\Big),
\end{align*}
where we have used bound \eqref{eq: bound for fk(n)}. Summing over $1\leq i<j\leq n$ and using Lemma \ref{lem: estimetes involving tau_3} to estimate the contribution of the above $O$-terms, we obtain 
$$
\sum_{1\leq i< j\leq n}\mathbb{E}(X_iX_j)=\sum_{1\leq i< j\leq n}f_k(i)f_k(j)+O_{\alpha,\varepsilon}(n^{3/2+\varepsilon})
$$
for any $\varepsilon>0$. Note that
\begin{align*}
2\sum_{1\leq i< j\leq n}f_k(i)f_k(j)=\big(\sum_{1\leq i\leq n}f_k(i)\big)^2-\sum_{1\leq i\leq n}f_k^2(i)=\big(\sum_{1\leq i\leq n}f_k(i)\big)^2+O_{\varepsilon}(n^{1+\varepsilon}), 
\end{align*}
where we have used
$$
\sum_{1\leq i\leq n}f_k^2(i)\ll\sum_{1\leq i\leq n}\tau_3^2(i)\ll_{\varepsilon} n^{1+\varepsilon}.
$$
Then we have
\begin{align}\label{eq: sum E(XiXj)=}
\sum_{1\leq i< j\leq n}\mathbb{E}(X_iX_j)=\frac{1}{2}\Big(\sum_{1\leq i\leq n}f_k(i)\Big)^2+O_{\alpha,\varepsilon}(n^{3/2+\varepsilon}).
\end{align}
Now Proposition \ref{prop: variance of Sn} follows from combining \eqref{eq:E(X_i^2)}, \eqref{eq: sum E(XiXj)=} and \eqref{eq: sum E^2sum Xi=} with \eqref{eq:V(S_n)}.

\section{Proof of Theorem \ref{thm:T_n=}}
The frame of the proof of Theorem \ref{thm:T_n=} is similar to the proof of Theorem \ref{thm:S_n=}. According to Lemma \ref{lemma:sum}, it is also sufficient to compute the expectation and estimate the variance of $\overline{T}_n$, respectively. We first prove the following counterpart of Lemma \ref{lem: mean of C_alpha(m,n) with k-free condition}.

\begin{lemma}\label{lemma: two k-free}
For $\alpha\in(0,1)$ and any integers $a_1,a_2,b_1,b_2,n,k$ with $b_1,b_2,n,k\in\mathbb{N}$ and $\gcd(b_1,b_2)=1$, we have 
$$
\sum_{\substack{0\leq m\leq n\\ \gcd(m+a_1,b_1)\ {\rm{is}}\ k-{\rm{free}}\\\gcd(m+a_2,b_2)\ {\rm{is}}\ k-{\rm{free}}}}C_{\alpha}(n,m)=f_k(b_1)f_k(b_2)+O_\alpha\Big(\frac{\tau_3(b_1)\tau_3(b_2)}{\sqrt{n}}\Big), 
$$
where where $C_{\alpha}(n,m)$ is given by \eqref{eq: C_alpha(n,s)=} and $f_k$ is given by \eqref{eq: definition of f_k(n)}. 
\end{lemma}
\begin{proof}
For simplicity, denote the left hand side of the above equation by $R$. Then we have
$$
R:=\sum\limits_{\substack{d_1\mid b_1,\ d_2\mid b_2\\ d_1,\ d_2{~{\rm are}~} k-{\rm free}}}\sum_{\substack{0\leq m\leq n\\ \gcd(a_1+m,b_1)=d_1\\ \gcd(a_2+m,b_2)=d_2}}C_{\alpha}(n,m).
$$
By Lemma \ref{lemma:mean1}, we obtain
$$
R=\sum\limits_{\substack{d_1\mid b_1,\ d_2\mid b_2\\ d_1,\ d_2{~{\rm are}~} k-{\rm free}}}\bigg(\frac{1}{d_1d_2}\sum\limits_{\substack{r_1\mid (b_1/d_1)\\ r_2\mid (b_2/d_2)}}\frac{\mu(r_1)\mu(r_2)}{r_1r_2}+O_\alpha\Big(\frac{\tau_2(b_1/d_1)\tau_2(b_2/d_2)}{\sqrt{n}}\Big)\bigg).
$$
The contribution of the $O$-term to $R$ is
$$
\ll_\alpha \frac{1}{\sqrt{n}}\Big(\sum\limits_{d_1\mid b_1}\tau_2(b_1/d_1)\Big)\Big(\sum\limits_{d_2\mid b_2}\tau_2(b_2/d_2)\Big)=\frac{\tau_3(b_1)\tau_3(b_2)}{\sqrt{n}}.
$$
Then we have
$$
R=\bigg(\sum\limits_{\substack{r_1d_1\mid b_1\\ d_1{~{\rm is}~} k-{\rm free}}}\frac{\mu(r_1)}{r_1d_1}\bigg)\bigg(\sum\limits_{\substack{r_2d_2\mid b_2\\ d_2{~{\rm is}~} k-{\rm free}}}\frac{\mu(r_2)}{r_2d_2}\bigg)+O_{\alpha}\Big(\frac{\tau_3(b_1)\tau_3(b_2)}{\sqrt{n}}\Big),
$$
which gives our desired result.
\end{proof}
\subsection{The expectation of $\overline{T}_n$}
In this subsection, we compute the expectation $\mathbb{E}(\overline{T}_n)$ and prove the following result.

\begin{proposition}\label{prop: mean of Tn}
For $\alpha\in (0,1)$ and positive integers $k,n$, we have 
$$
\mathbb{E}(\overline{T}_n)=\prod_{p}\Big(1-\frac{2}{p^{2k}}\Big)+O_{k,\alpha}(n^{-1/2+\varepsilon})
$$
for any $\varepsilon>0$.
\end{proposition}

By the definition of $\overline{T}_n$ and the linearity of the expectation, we have
\begin{align}\label{eq:E(T_n)}
	\mathbb{E}(\overline{T}_n)=\frac{1}{n}\sum_{1\leq i\leq n}\mathbb{E}(X_iX_{i+1}). 
\end{align}
For each $1\leq i\leq n$, by the definition of $X_i$, we have 
 \begin{align}\label{eq:E(XiXi+1)1}
 \mathbb{E}(X_iX_{i+1})=\mathbb{P}(P_i,P_{i+1}\ {\rm{are\ both}}\  k-{\rm{free}}),
 \end{align}
where $P_i,\ P_{i+1}$ are the $i$-th and $(i+1)$-th steps of the $\alpha$-random walker, respectively. It is the same as the expectation of $\overline{S}_n$. We denote $P_i=(l,i-l)$ for $0\leq l\leq i$. Thus we have $P_{i+1}=(l+m,i+1-l-m)$ for $m\in \{0,1\}$. The probability that $P_i,\ P_{i+1}$ both are $k$-free is of the form 
$$
\sum_{\substack{0\leq l\leq i\\ \gcd(l,i-l)\ {\rm{is}}\ k-{\rm{free}}}}C_{\alpha}(i,l)\sum_{\substack{0\leq m\leq 1\\ \gcd(l+m,i+1-l-m)\ {\rm{is}}\ k-{\rm{free}}}}C_{\alpha}(1,m). 
$$
Since $\gcd(l,i-l)=\gcd(l,i)$ and $\gcd(l+m,i+1-l-m)=\gcd(l+m,i+1)$, by \eqref{eq:E(XiXi+1)1}, we have 
\begin{align}\label{eq: E(X_iX_{i+1})=} 
	\mathbb{E}(X_iX_{i+1})=\sum_{\substack{0\leq l\leq i\\ \gcd(l,i)\ {\rm{is}}\ k-{\rm{free}}}}C_{\alpha}(i,l)\sum_{\substack{0\leq m\leq 1\\ \gcd(l+m,i+1)\ {\rm{is}}\ k-{\rm{free}}}}C_{\alpha}(1,m),
\end{align}
which gives
\begin{align}\label{eq:E(X_iX_i+1)2}
\mathbb{E}(X_iX_{i+1})=(1-\alpha)M_1+\alpha M_2,
\end{align}
where
\begin{align*}
M_1=\sum_{\substack{0\leq l\leq i\\ \gcd(l,i)\ {\rm{is}}\ k-{\rm{free}}\\ \gcd(l,i+1)\ {\rm{is}}\ k-{\rm{free}}}}C_{\alpha}(i,l)
\end{align*}
and
\begin{align*}
M_2=\sum_{\substack{0\leq l\leq i\\ \gcd(l,i)\ {\rm{is}}\ k-{\rm{free}}\\ \gcd(l+1,i+1)\ {\rm{is}}\ k-{\rm{free}}}}C_{\alpha}(i,l). 
\end{align*}
We apply Lemma \ref{lemma: two k-free} to $M_1$ and obtain 
\begin{align}\label{eq:M_1}
	M_1=f_k(i)f_k(i+1)+O_\alpha\Big(\frac{\tau_3(i)\tau_3(i+1)}{\sqrt{i}}\Big).
\end{align}
Similarly, by Lemma \ref{lemma: two k-free} again, we also have 
\begin{align}\label{eq:M_2}
	M_2=f_k(i)f_k(i+1)+O_\alpha\Big(\frac{\tau_3(i)\tau_3(i+1)}{\sqrt{i}}\Big).
\end{align}
Inserting \eqref{eq:M_1} and \eqref{eq:M_2} to \eqref{eq:E(X_iX_i+1)2} and with help of bound \eqref{eq: bound for tau_l(n)}, we  obtain
\begin{align}\label{eq:E(X_iX_i+1)3}
\mathbb{E}(X_iX_{i+1})=f_k(i)f_k(i+1)+O_{\alpha,\varepsilon}(i^{-1/2+\varepsilon})
\end{align}
for any $\varepsilon>0$. We then have 
\begin{align}\label{eq:4the sum of E(X_iX_{i+1})}
\sum_{1\leq i\leq n}\mathbb{E}(X_iX_{i+1})=\sum_{1\leq i\leq n}f_k(i)f_k(i+1)+O_{\alpha,\varepsilon}(n^{1/2+\varepsilon}).
\end{align}
Inserting this to \eqref{eq:E(T_n)} and applying Lemma \ref{lem: mean of f_k(n)f_k(n+1)} to the sum over $i$, we obtain our desired result.

\subsection{Estimating the variance of $\overline{T}_n$}

In this subsection, we estimate the variance of $\overline{T}_n$ and prove the following result.

\begin{proposition}\label{prop: variance of Tn}
For any $\alpha\in (0,1)$, $\varepsilon>0$ and integer $\ k\geq 1$ , we have
$$
\mathbb{V}(\overline{T}_n)=O_{k,\alpha}(n^{-1/2+\varepsilon}).
$$
\end{proposition}

By the definition of $\overline{T}_n$, we have 
\begin{align*}
	\mathbb{V}(\overline{T}_n)=\frac{1}{n^2}\Big(\sum_{1\leq i,j\leq n}\mathbb{E}(X_iX_{i+1}X_jX_{j+1})-\mathbb{E}^2(\sum_{1\leq i \leq n}X_iX_{i+1})\Big).
\end{align*}
It follows that
\begin{align}\label{eq:V(T_n)}
\mathbb{V}(\overline{T}_n)=\frac{1}{n^2}\sum_{1\leq i \leq n}\mathbb{E}(X_i^2X_{i+1}^2)+\frac{2}{n^2}\sum_{1\leq i<j \leq n}\mathbb{E}(X_iX_{i+1}X_jX_{j+1})-\frac{1}{n^2}\mathbb{E}^2\Big(\sum_{1\leq i \leq n}X_iX_{i+1}\Big).
\end{align}

For the first term on the right hand side of \eqref{eq:V(T_n)}, by the definition of $X_i$, we have 
\begin{align}\label{eq:2the sum of E(X_i^2X_{i+1}^2)}
\sum_{1\leq i \leq n}\mathbb{E}(X_i^2X_{i+1}^2)=\sum_{1\leq i \leq n}\mathbb{E}(X_iX_{i+1})=O(n), 
\end{align}
where we have used $\mathbb{E}(X_iX_{i+1})=\mathbb{P}(P_i,P_{i+1}{\rm ~are~both~ }k-{\rm free})=O(1)$ for any $1\leq i\leq n$.

For the third term on the right hand side of \eqref{eq:V(T_n)}, by \eqref{eq:4the sum of E(X_iX_{i+1})}, we have 
\begin{align}\label{eq:5E^2(the sum of X_iX_{i+1})}
\mathbb{E}^2(\sum_{1\leq i \leq n}X_iX_{i+1})=\Big(\sum_{1\leq i \leq n}f_k(i)f_k(i+1)\Big)^2+O_{\alpha,\varepsilon}(n^{3/2+\varepsilon})
\end{align}
for any $\varepsilon>0$, where we have used Lemma \ref{lem: mean of f_k(n)f_k(n+1)} to obtain the $O$-term in the above.

Now we deal with the second term on the right hand side of \eqref{eq:V(T_n)}. Note that
\begin{align}\label{eq:V(T_n)2}
\sum_{1\leq i<j \leq n}\mathbb{E}(X_iX_{i+1}X_jX_{j+1})=\sum_{1\leq i \leq n}\sum_{i+1<j\leq n}\mathbb{E}(X_iX_{i+1}X_jX_{j+1})+O(n),
\end{align}
since $\mathbb{E}(X_iX_{i+1}X_{i+2})=\mathbb{P}(P_i,\ P_{i+1},\ P_{i+2}\ {\rm are}\ k-{\rm{free}})=O(1)$.
For $i+1<j$, we have
\begin{align*}
\mathbb{E}&(X_iX_{i+1}X_jX_{j+1})=\sum_{\substack{0\leq l\leq i\\\gcd(l,i)\ {\rm{is}}\ k-{\rm{free}}}}C_{\alpha}(i,l)\sum_{\substack{0\leq m_1\leq 1\\ \gcd(l+m_1,i+1)\ {\rm{is}}\ k-{\rm{free}}}}C(1,m_1)\\
&\sum_{\substack{0\leq m_2\leq j-i-1\\ \gcd(l+m_1+m_2,j)\ {\rm{is}}\ k-{\rm{free}}}}C(j-i-1,m_2)\sum_{\substack{0\leq m_3\leq 1\\ \gcd(l+m_1+m_2+m_3,j+1)\ {\rm{is}}\ k-{\rm{free}}}}C(1,m_3).\nonumber 
\end{align*}
Writing the inner sum over $m_3$ explicitly, we have
\begin{align}\label{eq:E(X_iX_{i+1}X_jX_{j+1})1}
\mathbb{E}(X_iX_{i+1}X_jX_{j+1})=&\sum_{\substack{0\leq l\leq i\\\gcd(l,i)\ {\rm{is}}\ k-{\rm{free}}}}C_{\alpha}(i,l)\sum_{\substack{0\leq m_1\leq 1\\ \gcd(l+m_1,i+1)\ {\rm{is}}\ k-{\rm{free}}}}C_{\alpha}(1,m_1)\big((1-\alpha)M_3+\alpha M_4\big),
\end{align}
where
\begin{align*}
M_3=\sum_{\substack{0\leq m_2\leq j-i-1\\ \gcd(l+m_1+m_2,j)\ {\rm{is}}\ k-{\rm{free}}\\ \gcd(l+m_1+m_2,j+1)\ {\rm{is}}\ k-{\rm{free}}}}C_{\alpha}(j-i-1,m_2)
\end{align*}
and
\begin{align*}
M_4=\sum_{\substack{0\leq m_2\leq j-i-1\\ \gcd(l+m_1+m_2,j)\ {\rm{is}}\ k-{\rm{free}}\\\gcd(l+m_1+m_2+1,j+1)\ {\rm{is}}\ k-{\rm{free}}}}C_{\alpha}(j-i-1,m_2).
\end{align*}
Applying Lemma \ref{lemma: two k-free} to the sum $M_3$ over $m_2$, we obtain 
\begin{align}\label{eq:M_3}
	M_3=f_k(j)f_k(j+1)+O_{\alpha}\Big(\frac{\tau_3(j)\tau_3(j+1)}{\sqrt{j-i-1}}\Big). 
\end{align}
Similarly, applying Lemma \ref{lemma: two k-free} again to the sum $M_4$ over $m_2$, we also obtain 
\begin{align}\label{eq:M_4}
	M_4=f_k(j)f_k(j+1)+O_{\alpha}\Big(\frac{\tau_3(j)\tau_3(j+1)}{\sqrt{j-i-1}}\Big). 
\end{align}
Inserting \eqref{eq:M_3}, \eqref{eq:M_4} to \eqref{eq:E(X_iX_{i+1}X_jX_{j+1})1}, we obtain
\begin{align*}
\mathbb{E}(X_iX_{i+1}X_jX_{j+1})=&\sum_{\substack{0\leq l\leq i\\\gcd(l,i)\ {\rm{is}}\ k-{\rm{free}}}}C_{\alpha}(i,l)\sum_{\substack{0\leq m_1\leq 1\\ \gcd(l+m_1,i+1)\ {\rm{is}}\ k-{\rm{free}}}}C_{\alpha}(1,m_1)\\
&\Big(f_k(j)f_k(j+1)+O_{\alpha}\Big(\frac{\tau_3(j)\tau_3(j+1)}{\sqrt{j-i-1}}\Big)\Big).
\end{align*}
By the binomial theorem and bound \eqref{eq: bound for tau_l(n)}, we have
$$
\mathbb{E}(X_iX_{i+1}X_jX_{j+1})=f_k(j)f_k(j+1)M_5+O_{\alpha}\Big(\frac{\tau_3(j)\tau_3(j+1)}{\sqrt{j-i-1}}\Big),
$$
where
$$
M_5=\sum_{\substack{0\leq l\leq i\\ \gcd(l,i)\ {\rm{is}}\ k-{\rm{free}}}}C_{\alpha}(i,l)\sum_{\substack{0\leq m_1\leq 1\\ \gcd(l+m_1,i+1)\ {\rm{is}}\ k-{\rm{free}}}}C_{\alpha}(1,m_1). 
$$
Note that $M_5$ is the same as $E(X_iX_{i+1})$ given by $\eqref{eq: E(X_iX_{i+1})=}$. Then by \eqref{eq:E(X_iX_i+1)3} we have 
\begin{align*}
M_5=f_k(i)f_k(i+1)+O_{\alpha}\Big(\frac{\tau_3(i)\tau_3(i+1)}{\sqrt{i}}\Big). 
\end{align*}
It follows that
\begin{align*}
\mathbb{E}(X_iX_{i+1}X_jX_{j+1})=&f_k(i)f_k(i+1)f_k(j)f_k(j+1)+O_{\alpha}\Big(\frac{\tau_3(i)\tau_3(i+1)\tau_3(j)\tau_3(j+1)}{\sqrt{i}}\Big)\\
&+O_{\alpha}\Big(\frac{\tau_3(j)\tau_3(j+1)}{\sqrt{j-i-1}}\Big),
\end{align*}
where we have used
$$
f_k(j)f_k(j+1)\ll \tau_3(j)\tau_3(j+1). 
$$
Summing over $1\leq i\leq n$ and $i+1<j\leq n$, we obtain
\begin{align*}
\sum_{1\leq i \leq n}\sum_{i+1<j\leq n}\mathbb{E}(X_iX_{i+1}X_jX_{j+1})=&V_k(n)+O_{\alpha}\Big(\sum_{1\leq i \leq n}\sum_{i+1<j\leq n}\frac{\tau_3(i)\tau_3(i+1)\tau_3(j)\tau_3(j+1)}{\sqrt{i}}\Big)\\
&+O_{\alpha}\Big(\sum_{1\leq i \leq n}\sum_{i+1<j\leq n}\frac{\tau_3(j)\tau_3(j+1)}{\sqrt{j-i-1}}\Big),\nonumber
\end{align*}
where
$$
V_k(n)=\sum_{1\leq i \leq n}\sum_{i+1<j\leq n}f_k(i)f_k(i+1)f_k(j)f_k(j+1).
$$
By Lemma \ref{lem: estimetes (2) involving tau_3}, the contribution of $O$-terms is $O_\varepsilon(n^{3/2+\varepsilon})$  for any $\varepsilon>0$. Then we have
$$
\sum_{1\leq i \leq n}\sum_{i+1<j\leq n}\mathbb{E}(X_iX_{i+1}X_jX_{j+1})=V_k(n)+O_\varepsilon(n^{3/2+\varepsilon}).
$$
For $V_k(n)$, replacing $i+1<j\leq n$ by $i<j\leq n$ up to an error term $\ll_{\varepsilon} n^{1+\varepsilon}$, we have
$$
2V_k(n)=2\sum_{1\leq i<j\leq n}f_k(i)f_k(i+1)f_k(j)f_k(j+1)+O_\varepsilon(n^{1+\varepsilon}),
$$
which gives
\begin{align*}
	2V_k(n)&=\Big(\sum_{1\leq i \leq n}f_k(i)f_k(i+1)\Big)^2-\sum_{1\leq i \leq n}\big(f_k(i)f_k(i+1)\big)^2+O_\varepsilon(n^{1+\varepsilon})\\
	&=\Big(\sum_{1\leq i \leq n}f_k(i)f_k(i+1)\Big)^2+O_\varepsilon(n^{1+\varepsilon}). 
\end{align*}
Hence, we have 
\begin{align}\label{eq:4the sum of E(X_iX_{i+1}X_jX_{j+1})}
\sum_{1\leq i \leq n}\sum_{i+1<j\leq n}\mathbb{E}(X_iX_{i+1}X_jX_{j+1})=\frac{1}{2}\Big(\sum_{1\leq i \leq n}f_k(i)f_k(i+1)\Big)^2+O_{\varepsilon}(n^{3/2+\varepsilon}).
\end{align}
Inserting \eqref{eq:2the sum of E(X_i^2X_{i+1}^2)}, \eqref{eq:5E^2(the sum of X_iX_{i+1})} and \eqref{eq:4the sum of E(X_iX_{i+1}X_jX_{j+1})} to \eqref{eq:V(T_n)}, we obtain
\begin{align*}
\mathbb{V}(\overline{T}_n)=O_{k,\alpha}(n^{-1/2+\varepsilon})
\end{align*}
for any $\varepsilon>0$. This completes our proof.

\end{document}